%% file: first_draft.tex
\documentclass[a4paper,reqno]{amsart}
\usepackage{a4}
\usepackage{amsmath}
\usepackage{appendix}
\usepackage{enumitem}
\usepackage{amssymb}
\usepackage{mathrsfs}
\usepackage[applemac]{inputenc}
\usepackage{graphicx}
\usepackage{color}
\usepackage{algorithm2e}
\usepackage{subfig}
\usepackage{multirow}
\usepackage{array}
\usepackage{float}
\usepackage[%pdftex,
pdfborder={0 0 0},
colorlinks=true,
linkcolor=black,
citecolor=red,
pagebackref=true,
]{hyperref}

\def\a{a}

\newcommand{\ydos}{\mathbf{y_{2}}}

\newcommand{\M}{{\mathbb M}}

\newcommand{\ds}{\displaystyle}
\newcommand{\supp}{\mathrm{supp}\;}

\newtheorem{remark}{\textbf{Remark}}[section]

\newtheorem{lemma}{\textbf{Lemma}}[section]
\newtheorem{theorem}{\textbf{Theorem}}[section]

\newtheorem{proposition}{\textbf{Proposition}}[section]
\newtheorem{problem}{\textbf{Problem}}[section]

\numberwithin{equation}{section}

\title[Approximation of deterministic mean field games with control-affine dynamics]{Approximation of deterministic mean field games under polynomial growth conditions on the data} 

\author[Justina Gianatti]{Justina Gianatti\jg}
\thanks{\jg CIFASIS-CONICET-UNR, Ocampo y Esmeralda, S2000EZP, Rosario, Argentina (gianatti@cifasis-conicet.gov.ar).}
\author[Francisco J. Silva]{Francisco J. Silva\fs} 
\thanks{\fs Institut de recherche XLIM-DMI, UMR-CNRS 7252, Facult\'e des Sciences et Techniques, 
Universit\'e de Limoges, 87060 Limoges, France (francisco.silva@unilim.fr).}
\author[Ahmad Zorkot]{Ahmad Zorkot\ha}
\thanks{\ha Institut de recherche XLIM-DMI, UMR-CNRS 7252, Facult\'e des Sciences et Techniques, 
Universit\'e de Limoges, 87060 Limoges, France (ahmad.zorkot@unilim.fr).}
\newcommand{\jg}{$^1$} \newcommand{\fs}{$^2$}\newcommand{\ha}{$^3$}

\input mymacro
\begin{document}
\begin{abstract} 
We consider a deterministic mean field games problem in which a typical agent solves an optimal control problem where the dynamics is affine with respect to the control and the cost functional has a growth which is polynomial with respect to the state variable. In this framework, we construct a mean field game problem in discrete time and finite state space that approximates equilibria of the original game. Two numerical examples, solved with the fictitious play method, are presented.
\end{abstract}
\maketitle

{\small
\noindent {\bf AMS-Subject Classification:} 91A16, 	49N80, 35Q89, 65M99, 91A26.  \\[0.5ex]
}
%-----------------------------------------------------
\section{Introduction}

The theory of Mean Field Games (MFGs for short) problems has been introduced  in~\cite{MR2269875,MR2271747,MR2295621}
and, independently, in~\cite{HMC06} in order to describe the asymptotic behaviour of Nash equilibria of non-cooperative symmetric  differential games with a large number of indistinguishable  players, which, individually, have a minor influence on the game. The reader is referred to~\cite{MR3195844,MR3559742,MR3752669,MR3753660,MR4214773}, and the references therein, for an  overview on MFGs theory including their numerical approximation and applications in crowd motion, economics, and finance. Equilibria in MFGs are usually described in terms of a system of two equations, called MFG {\it system}; a Hamilton-Jacobi-Bellman equation, describing the optimal cost of a {\it typical player}, and a Fokker-Planck equation, describing the evolution of the players. 

This work deals with the numerical approximation of deterministic mean field games problems, i.e. when the underlying differential game is deterministic. In this setting, a relaxed, also called {\it Lagrangian}, formulation of equilibria involving a fixed point problem on the space of probability measures over the  trajectories of the players, has been introduced in \cite{MR3644590,CH17,cannarsa_capuani_2018}. We assume that the controlled dynamics of a  typical player in the MFG  has the form 
\be
\label{eq:state_equation_introduction}
\dot{\gamma}(t)= A(t,\gamma(t))+ B(t,\gamma(t))\alpha(t) \quad \text{for a.e. } t\in ]0,T[.
\ee
Here, $T>0$ denotes the time horizon,  $\gamma$ and $\alpha$, which take values in $\RR^{d}$ and in $\RR^{r}$, respectively, denote the state and the control of a typical player, and $A\colon [0,T]\times\RR^{d}\to\RR^{d}$ and $B\colon [0,T]\to\RR^{d\times r}$ are given functions.  When the typical player controls its velocity, i.e. $\dot{\gamma}(t)=\alpha(t)$ and MFG equilibria are characterized in terms of the MFG system, the reference~\cite{MR2928379} proposes a semi-discrete scheme which is shown to converge towards a solution to the MFG system. An implementable version of this approximation, including also a discretization of the space variable, has been introduced in~\cite{MR3148086} and it is shown to converge when the space dimension is equal to one. In the same framework, in~\cite{Carlini_Silva_Zorkot_2023}  the authors propose a Lagrange-Galerkin scheme for the continuity equation appearing in the MFG system and they show the convergence of a fully-discrete approximation in general state dimensions. By adopting the relaxed formulation, in~\cite{MR4030259} an approximation of the MFG problem written in terms of a discrete time and finite state MFG~\cite{MR2601334}, hereafter called {\it finite} MFG, is shown to converge in general dimensions. Moreover, under a monotonicity assumption on the interaction cost terms (see~\cite[Section 2.3]{MR2295621}), an adaptation of the fictitious play method (see~\cite{Brown_51,Robinsonn_51}) is shown to converge and hence allows to rigorously approximate a solution to the finite MFG. Finally, in the work~\cite{Gianatti-Silva-arxiv} the authors approximate MFGs by finite MFGs when the dynamics of the typical player takes the general form~\eqref{eq:state_equation_introduction}. The convergence is established in general dimensions, the key point being a careful discretization of the underlying optimal control problem. In particular, the results in~\cite{Gianatti-Silva-arxiv} cover the approximation of MFGs where the typical player controls its acceleration (see~\cite{MR4102464,MR4132067}).

In all the references above, the assumptions on the dependence of the cost functional with respect to the state do not allow a  polynomial growth. In particular, the cost cannot depend quadratically on the state, which is a typical case considered in the applications. Our aim in this work, which is complementary with~\cite{Gianatti-Silva-arxiv}, is to cover this case. The main difficulty coming from a polynomial growth of the cost is that the value function of a typical player is not globally Lipschitz with respect to the state. In particular, the optimal feedback law is only locally bounded and hence a careful analysis is needed in order to construct a scheme for the value function with good stability properties. Under our assumptions, which include the independence of $B(t,\gamma(t))$ on $\gamma(t)$, the optimal feedback controls have a linear growth with respect to the state. This property still allows us to construct an approximation of the MFG problem where the time marginals of the distributions of the states of the agents are supported on a compact set which is independent of the discretization steps. Next, the analysis in~\cite{Gianatti-Silva-arxiv} applies and yields the convergence of the scheme as the discretization parameters vanish. As in \cite{MR4030259,Gianatti-Silva-arxiv}, we adopt in this work the relaxed formulation of the MFG equilibrium, the main point being that, in the convergence study of our approximations, compactness properties for the solutions to the scheme are easier to establish.

The remainder of this article is organized as follows. Section~\ref{sec:preliminaries} fixes the notations and the assumptions in this work. It also recalls the notion of Lagrangian MFG equilibrium and provides an existence and uniqueness result. Section~\ref{sec:val_function_approximation} is central in this work as it builds the scheme used to approximate the value function of a typical player. We explain the relation between this scheme and a standard semi-Lagrangian scheme (see~\cite{MR3341715}) and we provide a convergence result. In Section~\ref{sec:approximation_mfg}, we describe the finite MFG that approximates the continuous one and we present existence, uniqueness, and  convergence results. Finally, in Section~\ref{sec:numerical_result} we consider two numerical examples where the cost functional depends polynomially on the state variable and the interactions terms are monotone, which allows us to approximate the solutions to the finite MFG problems by using the fictitious play method.\smallskip

{\bf Acknowledgements.} F. J. Silva and A. Zorkot where partially supported by l'Agence Nationale de la Recherche (ANR), project  ANR-22-CE40-0010, and by KAUST through the subaward agreement ORA-2021-CRG10-4674.6.
For the purpose of open access, the authors have applied a CC-BY public copyright licence to any Author Accepted Manuscript (AAM) version arising from this submission.

\section{Preliminaries}   
\label{sec:preliminaries}
In what follows, $|\cdot|$ will denote  the infinity norm in $\RR^{d}$, and, given $R>0$,  $B_{\infty}(0,R)$ (respectively $\ov{B}_{\infty}(0,R)$)  will denote the corresponding open (respectively closed) ball centered at $0$ and of radius $R$. We denote by $\P(\RR^{d})$ the set of probability measures over $\RR^{d}$ and, for $\mu\in\P(\RR^{d})$, we set $\supp(\mu)$ for its support. We define
$\P_{1}(\RR^d)$ as the subset of $\P(\RR^{d})$ consisting on probability measures with finite first order moment, i.e. 
\be 
\label{def:p_1}
\P_{1}(\RR^{d})=\bigg\{\mu\in\P(\RR^{d})\,\Big|\,\int_{\RR^{d}}|x|\dd\mu(x)<\infty\bigg\},
\ee
which is endowed with the Wasserstein distance 
\be
\label{def:wasserstein_uno}
d_{1}(\mu_{1},\mu_{2})=\inf_{\mu\in\Pi(\mu_{0},\mu_{1})}\int_{\RR^{d}\times\RR^{d}}|x-y|\dd\mu(x,y)\quad\text{for all }\mu_{1},\,\mu_{2}\in\P_{1}(\RR^{d}),
\ee
where $\Pi(\mu_{0},\mu_{1})$ denotes the subset of $\P_{1}(\RR^{d}\times\RR^{d})$ of probability measures with first and second marginals given by $\mu_{1}$ and $\mu_{2}$, respectively. Given $\nu\in\P(\RR^d)$ and a Borel function $\Psi:\RR^d\to\RR^q$ ($q\in\NN$), the {\it push-forward} measure $\Psi\sharp\nu$, defined on the $\sigma$-algebra of Borel sets $\B(\RR^{q})$, is defined by
\be
\Psi\sharp\nu(A)=\nu(\Psi^{-1}(A))\quad\text{for all }A\in \B(\RR^q), 
\ee

Let $T>0$. The mean field game problem that we will deal with in this article is defined in terms of $\ell\colon [0,T] \times \RR^{r}\times\RR^{d}\times\P_{1}(\RR^{d})\to\RR$, $g\colon\RR^{d}\times\P_{1}(\RR^{d})\to\RR^{d}$, $A\colon [0,T] \times \RR^{r}\to\RR^{d}$, $B\colon [0,T]\to\RR^{d\times r}$, and $m_0^{*}\in\P_{1}(\RR^{d})$. We will consider the following assumptions: 
\begin{enumerate}[label={\bf(H\arabic*)}]
\item 
\label{h:h1} 
The functions $\ell$ and  $g$ are continuous. Moreover, there exist $p\in]1,\infty[$, $C_{\ell,1}$, $C_{\ell,2}$, $C_{\ell,3}$, $C_{\ell,4}\in]0,\infty[$ and $C_{g,1}$, $C_{g,2}$, $C_{g,3}\in]0,\infty[$ such that, for every $(t,a,x,\mu)\in [0,T] \times \RR^{r}\times\RR^{d}\times\P_{1}(\RR^{d})$, 
\begin{align}
C_{\ell,1}|a|^{p}-C_{\ell,2}&\leq\ell(t,a,x,\mu)\leq C_{\ell,3}(1+|a|^{p}+|x|^{p}),
\label{h:h1_i}\\
-C_{g,1}&\leq g(x,\mu)\leq C_{g,2}(1+|x|^{p}),
\label{h:h1_ii}\\
|\ell(t,a,x,\mu)-\ell(t,a,y,\mu)|&\leq C_{\ell,4}\big(1+|x|^{p-1}+|y|^{p-1}+|a|^{p-1})|x-y|\quad\text{for all }y\in\RR^{d},
\label{h:h1_iii}\\
|g(x,\mu)-g(y,\mu)|&\leq C_{g,3}(1+|x|^{p-1}+|y|^{p-1})|x-y|\quad\text{for all }y\in\RR^{d}.
\label{h:h1_iv}
\end{align}
\item 
\label{h:h2}
The following hold:
\begin{enumerate}[label={\rm (\roman*)}]
\item 
\label{h:h2_i} The functions $A$  and $B$ are continuous. 
\item
\label{h:h2_ii}
There exists $L_{A}\in]0,\infty[$ such that, for every $(t,x)\in [0,T] \times \RR^{d}$, 
$$
|A(t,x)-A(t,y)|\leq L_{A}|x-y|\quad\text{for all }y\in\RR^{d}.
$$
\item 
\label{h:h2_iii}
We have $r\leq d$ and there exists $\{i_1, \hdots, i_{r}\}\subset \{1,\hdots, d\}$ such that, for all $t\in[0,T]$, the rows $i_1, \hdots, i_r$ of $B(t)$ are linearly independent.
\end{enumerate} 
\item There exists $C^{*}\in]0,\infty[$ such that $\supp(m_{0}^{*})\subset\ov{B}_{\infty}(0,C^{*})$. 
\label{h:h3} 
\item 
\label{h:h4}
The following hold:\smallskip
\begin{enumerate}[label={\rm (\roman*)}]
\item 
\label{h:h4_i}
The function $\ell$ can be written as  
$$
\hspace{0.7cm}\ell(t,a,x,\mu)=\ell_0(t,a,x) + f(t,x,\mu)  \quad \text{for all } t\in [0,T],\, a\in \RR^{r},\, x\in\RR^d,\, \mu\in \P_1(\RR^d),
$$
where $\ell_0\colon [0,T]\times \RR^{r}\times \RR^d \to \RR$ satisfies \ref{h:h1} and $f\colon[0,T]\times \RR^d\times\P_1(\RR^d)\to \RR$ is continuous, bounded, and there exists $L_{f}>0$ such that, for all $(t,\mu)\in [0,T]\times \P_1(\RR^d)$,
$$
\hspace{0.4cm}|f(t,x,\mu) -f(t,y,\mu)|\leq L_{f} |x-y| \quad \text{for all }x,y\in \RR^d. 
$$
\item 
\label{h:h4_ii}
The functions $f(t,\cdot, \cdot)$, for all $t \in [0,T]$, and $g$  are {\it monotone}, i.e. for $\Psi=f(t,\cdot,\cdot),\,g$ it holds that 
\be
\label{monotonia_Phi}
\ds \int_{\RR^d}\big(\Phi(x,\mu_1)-\Phi(x,\mu_2)\big)\dd(\mu_1-\mu_2)(x)  \geq 0 \quad \text{for all $\mu_1$, $\mu_2\in \P_1(\RR^d)$}.
\ee
\end{enumerate}
\end{enumerate}

Assumptions~\ref{h:h1},~\ref{h:h2}, and~\ref{h:h3} ensure that both the MFG problem defined in Problem~\ref{mfg_problem} below, and its approximation, introduced in Section~\ref{sec:approximation_mfg}, admit at least one solution. Assumption~\ref{h:h4} plays an important role in the uniqueness of the equilibrium for both the MFG and its approximation and also in the proof of the convergence of a numerical method to solve the finite MFG problem (see~\cite{Gianatti-Silva-arxiv}). 

Note that \ref{h:h2} implies that 
\be
\label{eq:cota_A}
|A(t,x)|\leq C_{A}(1+|x|)  \quad \text{for all } (t,x) \in [0,T] \times \RR^d,
\ee
where 
$C_{A}=\max\{\max_{t\in [0,T]}|A(t, 0 )|,L_{A}\}$. 
In what follows, setting $|B(t)|$ for the matrix norm of $B(t)$ induced by the infinity-norm in $\RR^{r}$, we set $C_{B}=\sup_{t\in[0,T]}|B(t)|$,  which is finite by~\ref{h:h2}\ref{h:h2_i}.

Let us describe the MFG problem considered in this article. Given $x\in\cR^d$ and $m\in C\left([0,T];\mathcal{P}_1(\cR^d)\right)$, a typical player positioned at $x$ at time $t=0$ solves an optimal control problem of the form 
\be
\label{oc_problem}
\left\{\ba{l}
\ds \inf  \; \int_{0}^{T} \ell\left(s, \alpha(s),\gamma(s), m(s) \right)\dd s + g(\gamma(T),m(T)) \\[10pt]
\mbox{s.t. } \hspace{0.3cm} \dot{\gamma}(s)= A(s,\gamma(s))+ B(s)\alpha(s) \quad \text{for a.e. } s \in ]0,T[,\\[6pt]
\hspace{0.9cm} \gamma(0)=x, \\[6pt]
\hspace{0.9cm} \gamma\in W^{1,p}([0,T];\cR^d), \; \alpha\in L^{p}([0,T];\cR^r).
\ea\right. \tag{$OC_{x,m}$}
\ee

Note that assumption~\ref{h:h1} states that the cost functional in~\eqref{oc_problem} has a polynomial growth with respect to the state and control variables. In particular, our conditions on the cost functional are more general than those in~\cite{Gianatti-Silva-arxiv}. On the other hand, regarding the dynamics in~\eqref{oc_problem}, in~\cite{Gianatti-Silva-arxiv} the matrix $B$ can also depend on the state variable. 
 
Let us endow $\Gamma:=C\left([0,T];\cR^d\right)$ with the supremum norm $\|\cdot \|_{\infty}$ and, for all $t\in[0,T]$, define  $e_t\colon \Gamma\to\cR^d$ by $e_t(\gamma)=\gamma(t)$ for all $\gamma\in\Gamma$. Let us also set
$$
\P_{m_0^{*}}(\Gamma)=\{\xi\in\mathcal{P}_1\left(\Gamma\right)\,|\,e_0\sharp\xi = m_0^{*}\}.$$
The notion of equilibria that we consider is the {\it Lagrangian {\rm MFG} equilibria}, defined as a solution to the following problem:

\begin{problem}
\label{mfg_problem} 
Find $\xi^*\in \mathcal{P}_{m_0^{*}}\left(\Gamma\right)$ such that $[0,T] \ni t\mapsto e_t\sharp \xi^* \in \P_1(\RR^d)$ belongs to $C\left([0,T];\mathcal{P}_1(\cR^d)\right)$ and  for $\xi^*$-a.e. $\gamma^*\in\Gamma$ there exists $\alpha^*\in L^p([0,T];\cR^r)$ such that $(\gamma^*,\alpha^*)$ solves \eqref{oc_problem} with $x=\gamma^*(0)$ and $m(t)=e_t\sharp \xi^*$ for all $t\in[0,T]$.
\end{problem}

We have the following result.

\begin{theorem}
\label{th:mfg_existence_uniqueness}  
Assume that~\ref{h:h1},~\ref{h:h2}, and~\ref{h:h3} hold. Then Problem~\ref{mfg_problem} admits at least one solution. Moreover, if~\ref{h:h4} holds and for every $m\in C([0,T];\P_{1}(\RR^d))$ and $m_{0}^{*}$-a.e. $x\in\RR^d$ problem~\eqref{oc_problem} admits a unique solution, then the MFG equilibrium is unique.
\end{theorem}
\begin{proof}
The existence of at least one solution follows from Theorem~\ref{main_result} below while the uniqueness result can be shown by arguing exactly as in the proof of~\cite[Theorem 2.2]{Gianatti-Silva-arxiv}.
\end{proof}
\section{The value function of a typical player and its approximation}
\label{sec:val_function_approximation}
 
In this  section we fix $p\in]1,\infty[$. For every $(t,x)\in[0,T]\times\RR^{d}$ and $\alpha\in L^{p}([t,T];\RR^{r})$, note that~\ref{h:h2} implies that 
\be
\label{eq:state_equation}
\dot{\gamma}(s)= A(s,\gamma(s))+ B(s)\alpha(s) \quad \text{for a.e. } s\in ]t,T[,\quad \gamma(t)=x.
\ee
admits a unique solution $\gamma^{t,x,\alpha}\in W^{1,p}([0,T];\RR^{d})$. Given $m\in C([0,T];\P_{1}(\RR^{d}))$, set 
$$
J^{t,x}[m](\alpha)= \int_{t}^{T} \ell(s,\alpha(s),\gamma^{t,x,\alpha}(s),m(s))\dd s + g(\gamma(T) ,m(T))\quad\text{for all }\alpha\in L^{p}([t,T];\RR^{r}).
$$ 

The value function $v[m]\colon[0,T]\times\RR^{d}\to\RR$ is defined as 
\be 
\label{eq:value_function}
v[m](t,x)=\inf\big\{J^{t,x}[m](\alpha)\,|\,\alpha\in L^{p}([t,T];\RR^{r})\big\}\quad\text{for all }(t,x)\in[0,T]\times \RR^{d}.
\ee
\begin{proposition}
\label{prop:value_function}
Assume that~\ref{h:h1} and~\ref{h:h2} hold, let $m\in C([0,T];\P_{1}(\RR^d))$, and let $(t,x)\in [0,T]\times\RR^d$. Then there exists $\alpha^{*}\in L^{p}\big([t,T];\RR^{d}\big)$ such that $v[m](t,x)=J^{t,x}[m](\alpha^{*})$. Moreover, the following hold:  
\begin{enumerate}[label={\rm(\roman*)}]
\item 
\label{prop:value_function_i}
The exists  $C_{\text{{\rm Lip}}}>0$, independent of $(t,x,m)$, such that 
\be 
\label{eq:value_function_Lipschitz}
\big|v[m](t,x)-v[m](t,y)\big|\leq C_{\text{{\rm Lip}}}(1+|x|^{p-1}+|y|^{p-1})|x-y|\quad\text{for all }t\in [0,T],\,y\in\RR^d.
\ee
\item
\label{prop:value_function_ii}
There exists $C_{\text{{\rm b}}}>0$, independent of $(t,x,m)$, such that 
\be 
\label{eq:cota_control}
|\alpha^{*}(s)|\leq C_{\text{{\rm b}}}\big(1+|\gamma^{t,x,\alpha^{*}}(s)|\big)\quad\text{for a.e. }s\in [t,T].
\ee
\end{enumerate}
\end{proposition}
\begin{proof}  Let $(\alpha_{n})_{n\in\NN}\subset L^{p}([t,T];\RR^{r})$ be a minimizing sequence for the right-hand side of~\eqref{eq:value_function}. Estimate~\eqref{h:h1_i} and~\eqref{h:h1_ii} imply that $(\alpha_{n})_{n\in\NN}$ is bounded in $L^{p}([t,T];\RR^{r})$  and hence, up to some subsequence,  it converges weakly to some $\alpha^{*}\in L^{p}([t,T];\RR^{r})$. It follows from~\eqref{eq:state_equation} and \ref{h:h2}\ref{h:h2_i}-\ref{h:h2_ii} that $\gamma^{t,x,\alpha_{n}}$ converges uniformly in $[t,T]$ to $\gamma^{t,x,\alpha^{*}}$  and hence, by~\cite[Theorem~3.23]{dacorogna89}, we deduce that $v[m](t,x)=J^{t,x}[m](\alpha^{*})$. Denoting by $\alpha^{0}$  the  null control and $\gamma^{0}=\gamma^{t,x,\alpha_{0}}$, estimate ~\eqref{eq:cota_A} and Gr\"onwall's lemma imply that 
$\sup_{s\in[t,T]}|\gamma^{0}(s)|\leq e^{C_{A}T}(|x|+C_{A}T)$. In turn, it follows from~\eqref{h:h1_i},~\eqref{h:h1_ii},~\eqref{eq:cota_A}, and Gr\"onwall's inequality that 
\begin{multline}
\label{eq:cota_lp_control_0}
C_{\ell,1}\int_{t}^{T}|\alpha^{*}(s)| ^{p}\dd s\leq TC_{\ell,2}+J^{t,x}[m](\alpha^{0})+C_{g,1}\\
\leq\Big(C_{\ell,2}+C_{\ell,3}(1+ e^{pC_{A}T}(|x|+C_{A}T)^{p})\Big)T+C_{g,1} + C_{g,2}\Big(1+e^{pC_{A}T}(|x|+C_{A}T)^{p}\Big)
\end{multline}
and  hence there exists $\tilde{C}>0$, independent  of  $(t,x,m)$, such that 
\be 
\label{eq:alpha_star_bounded_lp_x}
\|\alpha^{*}\|_{L^{p}}\leq \tilde{C}(1+|x|).
\ee
In particular, it holds that  
\be 
\label{eq:value_function_bounded_lp}
v[m](t,x)=\inf\big\{J^{t,x}[m](\alpha)\,|\,\alpha\in L^{p}([t,T];\RR^{r}), \;
\|\alpha\|_{L^{p}}\leq \tilde{C}(1+|x|)\big\}\quad\text{for all }(t,x)\in[0,T]\times \RR^{d}.
\ee
Let us now show assertions \ref{prop:value_function_i}-\ref{prop:value_function_ii}. 

\ref{prop:value_function_i}:  Let $y\in\RR^{d}$ and set $\gamma^{*}=\gamma^{t,x,\alpha^{*}}$ and  $\tilde{\gamma}=\gamma^{t,y,\alpha^{*}}$. By standard arguments, it follows from~\ref{h:h2}\ref{h:h2_i}-\ref{h:h2_ii} and Gr\"onwall's lemma that 
\be \label{eq:grownwall_1}
\sup_{s\in[t,T]}|\gamma^{*}(s)|\leq C(1+|x|),
\quad
\sup_{s\in[t,T]}|\tilde{\gamma}(s)|\leq C(1+|x|+|y|), \quad \text{and}
\quad
\sup_{s\in [t,T]}|\tilde{\gamma}^{*}(s)-\gamma^{*}(s)|\leq C|x-y|,
\ee
%\begin{align}
%\sup_{s\in[t,T]}|\gamma^{*}(s)|&\leq C(1+|x|),
%\label{eq:grownwall_1}\\
%\sup_{s\in[t,T]}|\tilde{\gamma}(s)|&\leq C(1+|x|+|y|),
%\label{eq:grownwall_2}\\
%\sup_{s\in [t,T]}|\tilde{\gamma}^{*}(s)-\gamma^{*}(s)|&\leq C|x-y|,
%\label{eq:grownwall_3}
%\end{align}
for some $C>0$ independent of $(t,x,y,m)$. In turn, by~\eqref{h:h1_ii}, \eqref{h:h1_iii},
H\"older's inequality, and~\eqref{eq:alpha_star_bounded_lp_x}, we have 
\begin{align}
v[m](t,y)-v[m](t,x)&\leq \int_t^T\Big(\ell(s, \alpha^{*}(s), \tilde{\gamma}(s),m(s))-\ell(s, \alpha^{*}(s),\gamma^{*}(s),m(s))\Big)\dd s\nonumber\\
&\hspace{0.3cm}+g(\tilde{\gamma}(T),m(T))-g(\gamma^{*}(T),m(T))\nonumber\\
&\leq C_{\ell,4}\int_t^T \left(1+|\tilde{\gamma}(s)|^{p-1}+|\gamma^{*}(s)|^{p-1}+|\alpha^{*}(s)|^{p-1}\right)|\tilde{\gamma}(s)-\gamma^{*}(s)|\dd s\nonumber\\
&\hspace{0.3cm}+ C_{g,3}\left(1+|\tilde{\gamma}(T)|^{p-1}+|\gamma^{*}(T)|^{p-1}\right)|\tilde{\gamma}(T)-\gamma^{*}(T)|\\
&\leq C_{\text{{\rm Lip}}}(1+|x|^{p-1}+|y|^{p-1})|x-y|, 
\end{align}
for some $C_{\text{{\rm Lip}}}>0$ independent of $(t,x,y,m)$. The inequality for $v[m](t,x)-v[m](t,y)$ follows by exchanging the roles of $x$ and $y$ in the previous computation.

\ref{prop:value_function_ii}: Let $s\in[t,T[$, $h\in [0,T-s[$, and set $y^{*}=\gamma^{*}(s)$. Since $v[m]$ satisfies the dynamic programming inequality (see e.g.~\cite{MR1484411})
\be 
v[m](s,\gamma^{*}(s))\leq\int_{s}^{s+h}\ell(r,\gamma^{s,y^*,\alpha}(r),\alpha(r),m(r))\dd r+v[m]\big(s+h,\gamma^{s,y^*,\alpha}(s+h)\big),
\ee
for all $\alpha\in L^{p}([t,T];\RR^d)$, with equality for  $\alpha=\alpha^{*}$, by taking $\alpha=\alpha_0$ (the null control) and $\alpha=\alpha^{*}$,  the equality $v[m](s,y^{*})=J^{s,y^{*}}[m](\alpha^{*}|_{[s,T]})$, 
and estimate~\eqref{h:h1_i} yield
\begin{align}
C_{\ell,1}\int_{s}^{s+h}|\alpha^{*}(r)|^{p}\dd r&\leq C_{\ell,2}h+C_{\ell,3}\int_{s}^{s+h}\big(1+|\gamma^{s,y^{*},\alpha_{0}}(r)|^{p}\big)\dd r\nonumber\\
&\hspace{0.3cm}+v[m](s+h,\gamma^{s,y^{*},\alpha_{0}}(s+h))-v[m](s+h,\gamma^{*}(s+h)).
\label{eq:lipschitz_cuenta}
\end{align}
In what follows, $C>0$ will denote a constant independent of $(t,x,y,m)$ which may change from line to line.
Since \ref{h:h2}\ref{h:h2_i},~\eqref{eq:cota_A}, and  Gr\"onwall's lemma imply the existence of $C>0$ such that
\be
|\gamma^{*}(s+h)-\gamma^{s,y^{*},\alpha_{0}}(s+h)|\leq C\int_{s}^{s+h}|\alpha^{*}(r)|\dd r, 
\ee
we deduce from~\eqref{eq:lipschitz_cuenta},~\eqref{eq:grownwall_1},~\ref{prop:value_function_i}, 
 and  Young's inequality the existence of $C>0$ such that 
\be 
\int_{s}^{s+h}|\alpha^{*}(r)|^{p}\dd r\leq Ch(1+|y^{*}|^{p})
\ee
and, hence,~\eqref{eq:cota_control} follows from the Lebesgue differentiation theorem (see e.g.~\cite{MR2267655}).
\end{proof}
\begin{remark}
\label{rem:reescritura_de_v}
Proposition~\ref{prop:value_function}\ref{prop:value_function_ii} implies that, for any $(t,x)\in[0,T]\times\RR^{d}$,  $v[m](t,x)$ can be rewritten as 
\be 
\label{eq:value_function_bis}
v[m](t,x)=\inf\big\{J^{t,x}[m](\alpha)\,|\,\alpha\in L^{\infty}([t,T];\RR^{r}),\;|\alpha(s)|\leq C_{{\rm b}}(1+|\gamma^{t,x,\alpha}(s)|)\;\text{for a.e. }s\in [t,T]\big\}.
\ee
\end{remark}
We consider now the approximation of the value function $v[m]$ given by~\eqref{eq:value_function}. Let $N_t\in\NN$, $N_s\in \NN$, with $N_{s}\geq N_{t}$, and set $\Delta t=1/N_t$, $\Delta x=1/N_x$, $\I=\{0,\hdots,N_{t}\}$, $\I^{*}=\I\setminus\{N_{t}\}$, and $\G=\{i\Delta x\,|\,i\in\ZZ^{d}\}$. Given a regular mesh $\mathscr{T}$ with vertices in $\G$, let $(\psi_{x})_{x\in\G}$ be a $\QQ_1$ basis, i.e. for every $x\in\G$, $\psi_{x}$ is a nonnegative polynomial of partial degree less than or equal to $1$ on each element of $\mathscr{T}$, $\psi_{x}(x)=1$, $\psi_{x}(y)=0$ for all $y\in\G$ with $y\neq x$, and $\sum_{x\in\G}\psi_{x}(z)=1$ for all $z\in \RR^{d}$.  
Given  $\varphi\colon \G\to\RR$, define its interpolant $I[\varphi]\colon \RR^{d}\to\RR$ by
$$
I[\varphi](x)=\sum_{y \in \G}\psi_{y}(x)\varphi(y)  \quad \text{for all }x\in\RR^{d}.
$$
In view of Remark~\ref{rem:reescritura_de_v}, a standard semi-Lagrangian scheme (see e.g.~\cite{MR3341715}) to approximate $v[m]$ is given by 
\be
\label{eq:SL_fully_discreto_HJB}
\ba{rcl}
V_k(x)&=&\min\limits_{a\in \ov{B}_{\infty}(0,C_{{\rm b}}(1+|x|))} \bigg\{\Delta t \ell(t_k,a, x,m(t_k)) +I[V_{k+1}](\Phi(k,x,a))\bigg\} \\[12pt]
 \; & \; & \hspace{8.5cm}\text{for all } k\in \I^*,\,x\in \G, \\[-4pt]
V_{N_t}(x)&=& g(x,m(T)) \quad \text{for all }x\in \G, 
\ea  
\ee
where 
\be 
\label{def:y}
\Phi(k,x,a)=x+\Delta t(A(t_k,x)+B(t_k)a)\quad\text{for all }k\in\I^{*},\, x\in\G,\, a\in\RR^{r}.
\ee
We now introduce a variation of the previous scheme which exploits the particular structure of the dynamics in~\eqref{eq:state_equation}. First, notice that, by \ref{h:h2}\ref{h:h2_iii}, without loss of generality we can write
\be
\label{A_B_structure}
A(t,x)=\left(\ba{c}A_1(t,x)\\ 
A_2(t,x)\ea\right)\quad\mbox{and}\quad B(t)= \left( \ba{c} B_{1}(t) \\
B_2(t)\ea \right)\quad\mbox{for all }(t,x)\in [0,T]\times\RR^{d},
\ee 
where $A_1\colon [0,T]\times\cR^d\to\cR^r$, $A_2\colon [0,T]\times\cR^d\to \cR^{d-r}$, $B_{1}\colon  [0,T]\to \RR^{r \times r}$ is such that $B_1(t)$ is invertible for all $t \in [0,T]$, and $B_{2}\colon  [0,T]\to \RR^{(d-r) \times r}$. We partition the coordinates of $x\in\RR^{d}$ accordingly by writing $x=(\mathrm{x}_1, \mathrm{x}_2)$, where $\mathrm{x}_1\in \RR^{r}$ and $\mathrm{x}_2\in \RR^{d-r}$ denote the first $r$ and the last $d-r$ components of $x$, respectively. We also write $\G=\G_{r}\times\G_{d-r}$, where 
$$
\G_r =\left\{i\Delta x\;|\;i\in\ZZ^r\right\}\quad\text{and}\quad \G_{d-r} =\left\{i\Delta x\;|\;i\in\ZZ^{d-r}\right\},
$$
and we suppose that the basis $(\psi_{x})_{x\in \G}$ can be decomposed as the tensorial product of two $\QQ_{1}$ basis $(\eta_{\mathrm{x}_1})_{\mathrm{x}_1\in \G_{r}}$ and $(\beta_{\mathrm{x}_2})_{\mathrm{x}_2\in \G_{d-r}}$ 
defined on regular meshes with vertices in $\G_{r}$ and $\G_{d-r}$, respectively. More precisely, we suppose that for every $x=(\mathrm{x}_1,\mathrm{x}_2)\in\G$ we have
\be
\label{def:psi_tensor_product}
\psi_{x}(y)=\eta_{\mathrm{x}_{1}}(\mathrm{y}_{1})\beta_{\mathrm{x}_{2}}(\mathrm{y}_{2})\quad\text{for all }y=(\mathrm{y}_1,\mathrm{y}_2)\in\RR^{d}.
\ee
In what follows, we assume the  existence of $C_{I}>0$, independent of $\Delta x$, such  that 
\be
\label{support_beta_x}
\supp(\beta_{\mathrm{x}_{2}}) \subseteq  \{\mathrm{y}_{2}\in \RR^{d-r}\,|\, |\mathrm{y}_{2}-\mathrm{x}_{2}| \leq C_I \Delta x \}\quad\text{for all }\mathrm{x}_{2}\in\G_{d-r}.
\ee  
Let $k\in\I^{*}$ and $x\in\RR^{d}$. In the modified version of ~\eqref{eq:SL_fully_discreto_HJB}, we will only consider controls $a\in\ov{B}_{\infty}(0,C_{{\rm b}}(1+|x|))$ such that, writing 
$\Phi(k,x,a)=(\Phi_1(k,x,a),\Phi_2(k,x,a))$, we have $\Phi_1(k,x,a)\in \G_{r}$. Notice that, for every $\mathrm{y}_{1}\in\G_{r}$, it holds that
\be
\label{def:alpha_fullydiscrete}
\mathrm{y}_{1}=\Phi_1(k,x,a)\;\Leftrightarrow\; a=B_1(t_k)^{-1}\left[ \frac{\mathrm{y}_1-\mathrm{x}_1}{\Delta t}-A_1(t_k, x)\right].
\ee
Thus, setting
\be
\label{definition_alpha_n}
\ba{rcl}
\alpha (k,x, \mathrm{y}_1)&:=&B_1(t_k)^{-1}\left[ \frac{\mathrm{y}_1-\mathrm{x}_1}{\Delta t}-A_1(t_k, x)\right]\in \RR^r,\\[6pt]
\ydos(k,x,\mathrm{y}_1)&:= &  \mathrm{x}_2 +\Delta t \left[A_2(t_k, x)+B_2(t_k)\alpha(k, x, \mathrm{y}_1)\right]  \in\RR^{d-r},
\ea
\ee
it is natural to define the sets 
\be
\label{definicion_de_los_S}
\ba{rcl}
\SS_{k+1}^1(x)&=&\left\{\mathrm{y}_1\in\G_r\; \big|\; |\alpha (k,x, \mathrm{y}_1)|\leq  C_{{\rm b}}(1+|x|)\right\}, \\[6pt]
\SS_{k+1}^2(x,\mathrm{y}_1)&=&\left\{\mathrm{y}_2 \in\G_{d-r}\; \big|\; \ydos(k,x,\mathrm{y}_1)\in \supp{\beta_{\mathrm{y}_2}} \right\} \quad \text{for }  \mathrm{y}_1\in \SS_{k+1}^1(x),\\[6pt]
\SS_{k+1}(x)&=&\left\{(\mathrm{y}_1,\mathrm{y}_2)\in\G\; \big|\; \mathrm{y}_1\in\SS_{k+1}^1(x),\;      \mathrm{y}_2\in \SS_{k+1}^2(x,\mathrm{y}_1)\right\}.
\ea
\ee
Arguing as in~\cite[Lemma~3.2]{Gianatti-Silva-arxiv}, one checks that, if $\Delta x/\Delta t$ is small enough, then $\SS_{k+1}(x)\neq\emptyset$. Starting from an initial grid $\SS_{0}=\G\cap \ov{B}_{\infty}(0,C^{*})$, we can then construct the family of time-depending grids 
\be
\label{eq:family_s}
\SS_{k+1}=\bigcup_{x\in\SS_{k}}\SS_{k+1}(x)\quad\text{for all }k\in\I^{*}
\ee
with the property that $(\mathrm{y}_{1},\mathrm{y}_{2})\in \SS_{k+1}$ if and only if there exists $x\in\SS_{k}$ and $a\in\RR^{r}$ such that $|a|\leq C_{{\rm b}}(1+|x|)$, $\mathrm{y}_{1}=\Phi_1(k,x,a)$, and $\Phi_2(k,x,a)\in \supp{\beta_{\mathrm{y}_2}}$.

On the other hand, notice that, for every $\varphi\colon\G\to\RR$, $\mathrm{y}_{1}\in\G_{r}$, and $\mathrm{y}_{2}\in\RR^{d-r}$, we have 
\begin{multline}
\label{eq:reduction_interpolation}
I[\varphi](\mathrm{y}_{1},\mathrm{y}_{2})=\sum_{\mathrm{z}_{1}\in\G_{r},\mathrm{z}_{2}\in\G_{d-r}}\psi_{z}(\mathrm{y}_{1},\mathrm{y}_{2})\varphi(\mathrm{z}_{1},\mathrm{z}_{2})\\
=\sum_{\mathrm{z}_{1}\in\G_{r}}\eta_{z_{1}}(\mathrm{y}_{1})\sum_{\mathrm{z}_{2}\in\G_{d-r}}\beta_{\mathrm{z}_{2}
}(\mathrm{y}_{2})\varphi(\mathrm{z}_{1},\mathrm{z}_{2})
=\sum_{\mathrm{z}_{2}\in\G_{d-r}}\beta_{\mathrm{z}_{2}
}(\mathrm{y}_{2})\varphi(\mathrm{y}_{1},\mathrm{z}_{2}).
\end{multline}

Altogether,~\eqref{eq:SL_fully_discreto_HJB}-\eqref{eq:reduction_interpolation} suggest to consider the following variation of~\eqref{eq:SL_fully_discreto_HJB}:
\be
\label{fully_discreto_HJB}
\ba{rcl}
 \mathcal{V}_k(x) &=&\min\limits_{ \mathrm{y}_1\in  \SS^1_{k+1}(x)} \bigg\{\Delta t \ell(t_k, \alpha(k,x, \mathrm{y}_1), x,m(t_k)) +I_{\SS_{k+1}^2(x,\mathrm{y}_1)}[\mathcal{V}_{k+1}(\mathrm{y}_1,\cdot)](\ydos(k,x,\mathrm{y}_1))\bigg\}\\[12pt]
 \; & \; & \hspace{10cm}\text{for all } k\in \I^*,\,x\in \SS_{k}, \\[-4pt]
\mathcal{V}_{N_t}(x)&=& g(x,m(T)) \quad \text{for all }x\in \SS_{N_t},
\ea  
\ee
where, for $F\subseteq\G_{d-r}$ and $\varphi\colon F\to\RR$, we  have set 
$$
I_{F}[\varphi](\mathrm{y}_2)=\sum_{\mathrm{x}_2 \in F}\beta_{\mathrm{x}_2}(\mathrm{y}_2)\varphi(\mathrm{x}_2 )  \quad \text{for all }\mathrm{y}_2\in\RR^{d-r}.
$$
Notice that~\eqref{fully_discreto_HJB} can be rewritten as 
\be
\label{fully_discrete_HJB}
\ba{l}
\mathcal{V}_k(x) =\min\limits_{p \in\P(\SS^1_{k+1}(x))}\bigg\{\sum\limits_{\mathrm{y}_1 \in  \SS^1_{k+1}(x)} p(\mathrm{y}_1)\Big[\Delta t \ell(t_k, \alpha(k,x, \mathrm{y}_1), x,m(t_k)) \\[5pt]\hspace{5cm}+I_{\SS_{k+1}^2(x,\mathrm{y}_1)}[\mathcal{V}_{k+1}(\mathrm{y}_1,\cdot)]\big(\ydos(k,x,\mathrm{y}_1)\big)\Big]\bigg\}\quad \text{for all }k\in \I^*, \; x\in \SS_{k},  \\[2pt]
\mathcal{V}_{N_t}(x)= g(x,m(T))\quad\text{for all }x\in \SS_{N_t}.
\ea  
\ee

The following result, which shows that the family of time dependent grids $(\SS_{k})_{k\in\I}$ remains uniformly bounded with respect to the discretization parameters, will play a key role in what follows. 

\begin{lemma}
\label{the_big_compact} 
There exists a nonempty compact set $K\subset \RR^{d}$, independent of $\Delta t$ and $\Delta x$ as long as $\Delta x /\Delta t \leq 1$, such that 
$$
\SS_{k}\subset K\quad\text{for all }k\in \I^*.
$$
\end{lemma}
\begin{proof} 
Consider the following family of compact sets: set $K_0=\ov{B}_{\infty}(0,C^{*})$ and 
\be
\label{compacts_Ks}
K_{k+1} := K_{k} +\left(\Delta t\left[ \big(C_{A}+C_{B} C_{{\rm b}}\big)\big(1+\sup_{x\in K_{k}}|x|\big)\right]+C_{I}\Delta x\right) \ov{B}_{\infty}(0,1)  \quad \text{for all } k \in \I^*,
\ee
where we recall that $C_{A}$ is given in~\eqref{eq:cota_A}, $C_{B}=\sup_{t\in[0,T]}|B(t)|$, and $C_{I}$ satisfies~\eqref{support_beta_x}. It follows from~\eqref{definicion_de_los_S} and~\eqref{eq:family_s} that $\SS_{k}\subset K_{k}$ for all $k\in\I$ and hence it suffices to show that the family $(K_{k})_{k\in\I}$ is uniformly bounded.  Let $k\in \I$ and set $c_k=\sup_{x\in K_{k}} |x|$. Equation~ \eqref{compacts_Ks} yields
$$
\ba{rcl}
c_{k+1}&\leq & c_{k}+ \left(\Delta t\left[ \big(C_{A}+C_{B}C_{{\rm b}}\big)\big(1+c_{k}\big) +C_{I}\frac{\Delta x}{\Delta t}\right]\right) \\[6pt]
\; &\leq & \big(1+ \Delta t(C_{A}+C_{B}C_{{\rm b}})\big) c_{k} + \Delta t (C_{A}+ C_{B}C_{{\rm b}}  +C_{I}),
\ea
$$
which, by the discrete Gr\"onwall's lemma, implies that the set $\{c_{k}\,|\,k\in\I\}$ is uniformly bounded. The result follows.
\end{proof}

\begin{proposition} 
\label{prop:convergencia-value-function-discrete-continuous}
Assume that~\ref{h:h1} and~\ref{h:h2} hold. Consider three sequences $(N_t^{n},N^n_s)\subset \NN^{2}$  and $(m_{n})_{n\in\NN}\subset C([0,T];\P_{1}(\RR^{d}))$ such that, as $n\to\infty$,  $N_t^{n} \to \infty$, $N^n_s \to \infty$, $   N^n_t/N^n_s \to 0$, and $m_{n}\to m^{*}$ for some $m^{*}\in C([0,T];\P_{1}(\RR^{d}))$. Set $\I^{n}=\{0,\hdots,N^{n}_{t}\}$ and, associated with the parameters $(N_t^{n},N^n_s)$ and $m^{n}$, define  $\SS_k^{n}$ as in~\eqref{definicion_de_los_S} and  denote  by  $\V^{n}$ the solution to~\eqref{fully_discreto_HJB}. Then it holds that 
\be
\label{convergencia_value_function_HJB}
\sup \left\{ \left|\V_{k}^{n}(x) - v[m^{*}](t_k^{n},x) \right| \; \big| \; k\in \I^{n}, \;  x\in \SS_k^{n}\right\}\underset{n\to \infty}{\longrightarrow} 0,
\ee
where $v[m^{*}]$ is defined in~\eqref{eq:value_function}.
\end{proposition}
\begin{proof}[Sketch of the proof] Let $m\in C([0,T];\P_{1}(\RR^{d}))$ and consider, as an intermediate step, the following semi-discrete scheme to approximate $v[m]$: 
\be
\label{eq:semidiscrete-scheme-dpp}
\ba{rcl}
v_d(k,x)&=&\underset{\a\in\cR^r}\min \; \left\{\Delta t \ell(t_k,\a,x,m(t_k))+v_d[m]\left(k+1,x+\Delta t[A(t_k, x)+ B(t_k)\a]\right)\right\}\\[5pt]
&\,&\hspace{7cm} \text{for all }k\in \I^{*},\; x\in\cR^d, \\[2pt]
v_d(N_t,x)&=&g(x,m(T))\quad\text{for all }x\in\cR^d.
\ea
\ee
Setting $v_{d}^{n}$ for the solution to the previous scheme associated with $m^{n}$, using the framework developed in~\cite{BarSou91}, and arguing as in~\cite[Proposition 3.1]{Gianatti-Silva-arxiv} one shows that for every compact set $K\subset\cR^d$ it holds that   
\be
\label{eq:converegence_sd_continuo}
\sup_{(k,x)\in\I^n\times K}\left|v^n_d(k,x)-v[m^*](t_k^n,x)\right|\underset{n\rightarrow \infty}{\longrightarrow} 0. 
\ee

On the other hand, by adapting to the discrete case the proof of Proposition~\ref{prop:value_function}, one checks that the minimization on the right-hand-side of~\eqref{eq:semidiscrete-scheme-dpp} can be restricted the set $\{a\in\RR^{d}\,|\,|a|\leq C_{{\rm b}}(1+|x|)\}$. Using this fact, Lemma~\ref{the_big_compact}, and arguing as in the proof of \cite[Lemma 5.3 (ii)]{Gianatti-Silva-arxiv} we obtain that 

\be 
\label{eq:diff_discrete_fully}
\max\left\{\left|v^n_d(k,x)-\mathcal{V}^n_k(x) \right|\;|\;k\in\I^n,\;x\in\SS^n_k\right\} \to 0
\quad\text{as $n\to \infty$}, 
\ee
and hence~\eqref{convergencia_value_function_HJB} follows from~\eqref{eq:converegence_sd_continuo} and \eqref{eq:diff_discrete_fully}.
\end{proof}
\section{The finite mean field game approximation}
\label{sec:approximation_mfg}
In this section, given $N_t\in\NN$,  $N_s\in\NN$, with $N_s\geq N_t$, we approximate Problem~\ref{mfg_problem} by a fixed point problem of a map ${\bf br}$, called {\it best response mapping} defined on the space $\M=\prod_{k \in \I}\P(\SS_{k})$ of discrete time marginals. The resulting approximation will take the form of a discrete time and finite state MFG (see~\cite{MR2601334}). In order to construct the map ${\bf br}$, let us first introduce some useful definitions. For every $k\in\I$, we identify $p\in\P(\SS_{k})$ with the probability measure $\sum_{x\in\SS_{k}}p(\{x\})\delta_{x}\in\P_{1}(\RR^{d})$ and, given a finite set $F$,  the (nonpositive) entropy function $\mathcal{E}_F\colon \mathcal{P}(F)\to\cR$ is defined by
$$
\mathcal{E}_F(p) =\sum_{x\in F}p(x)\log(p(x)) \quad \text{for all }p\in\P(F),
$$
with the convention that  $p(x)=p(\{x\})$ and $0\log(0)=0$. Given $M\in\M$ and $\eps>0$, let us consider the following variation of~\eqref{fully_discrete_HJB}:
\be
\label{eq:hjb_eq}
\ba{l}
V_k^{M}(x)=\min\limits_{p\in \mathcal{P}( \SS^1_{k+1}(x))} \bigg\{\sum\limits_{\mathrm{y}_1\in  \SS_{k+1}^1(x)}p(\mathrm{y}_1)\bigg[ \Delta t  \ell(t_k, \alpha(k,x, \mathrm{y}_1), x, M_k) \\[6pt]
\hspace{3.2cm} +I_{\SS_{k+1}^2(x,\mathrm{y}_1)}[V_{k+1}^{M}(\mathrm{y}_1,\cdot)](\ydos(k,x,\mathrm{y}_1))\bigg] +\eps\mathcal{E}_{\SS^1_{k+1}(x)}(p)\bigg\} \quad \text{for all } k\in \I^*,\,x\in \SS_{k}, \\ [8pt]
 V_{N_t}^{M}(x)=g(x, M_{N_t}) \quad \text{for all }x\in \SS_{N_t}.
\ea
\ee

Notice that the incorporation of the entropy term in the scheme above implies that, for every $k\in\I^{*}$ and $x\in\SS_{k}$, the optimization problem defining $V_{k}^{M}(x)$ admits a unique solution $p_{k}^{M}(x,\cdot)$ which satisfies $p_{k}^{M}(x,\mathrm{y}_1)>0$ for all $\mathrm{y}_1\in\SS_{k+1}^{1}(x)$. Given $y\in \SS_{k+1}$, we also set 
\be\label{eq:def-P_k_M}
P_k^{M}(x, y):=  \begin{cases}p_k^{M}(x, \mathrm{y}_1) \beta_{\mathrm{y}_2}(\ydos(k,x,\mathrm{y}_1)) & \mbox{if $y \in \SS_{k+1}(x)$}
,\\[1mm]
0 &  \mbox{if $y \in \SS_{k+1} \setminus \SS_{k+1}(x)$}.
\end{cases}
\ee

Letting $E(x)=\left\{y\in\cR^d\;|\;|x-y|\leq \Delta x/2\right\}$ for all $x\in \G$, 
we define ${\bf br}(M)$ as the solution to
\be
\label{def-br}
\ba{rcl}
\widehat{M}_{k+1}(y)&=&\sum\limits_{x\in\SS_{k}}P_k^{M}(x,y)\widehat{M}_{k}(x)\quad\text{for all }k\in \I^*,\,y\in \SS_{k+1},\\ [13pt]
\hspace{0.34cm}\widehat{M}_0(x)&=&m_0^{*}(E(x))\quad\text{for all } x\in \SS_{0}.
\ea
\ee

The discretization of Problem \ref{mfg_problem} that we consider in this work reads as follows.
\begin{problem}
\label{MFG-SL_discrete_scheme_n-degenerate}  
Find $M\in \M$ such that $M={\bf br}(M)$.
\end{problem}

We have the following result. 
\begin{theorem} Assume that~\ref{h:h1},~\ref{h:h2}, and~\ref{h:h3} hold. Then Problem~\ref{MFG-SL_discrete_scheme_n-degenerate} admits at least one solution. In addition, if~\ref{h:h4} holds then the solution is unique.
\end{theorem}
\begin{proof}
Since, for every $M\in\M$, $k\in\I^{*}$, and $x\in\SS_{k}$ we have that $p_k^{M}(x,\cdot)$ is unique, it is easy to check that ${\bf br}$ is continuous. In turn, the existence of a fixed point of ${\bf br}$ follows from Brouwer's fixed point theorem. The uniqueness result follows from the arguments in the proof of \cite[Proposition~4.2]{Gianatti-Silva-arxiv}, the key point being that, if $\widehat{M}={\bf br}(M)$, then $\widehat{M}_k(x)>0$ for all $k\in\I$ and $x\in\SS_k$. 
\end{proof}

Now, let us discuss the convergence of solutions to~\eqref{MFG-SL_discrete_scheme_n-degenerate} towards a solution to Problem~\ref{mfg_problem} as the discretization parameters $\Delta t$, $\Delta x$, and $\eps$ tend to zero. Let $(N_t^n)_{n\in\NN}\subset\NN$, $(N_s^n)_{n\in\NN}\subset\NN$, $(\eps_n)_{n\in\NN}\subset ]0,\infty[$,  and, for every $n\in\NN$, set $\Delta t_n= T/N_{t}^{n}$, $\Delta x_n= 1/N_{s}^{n}$, $\I^n=\{0,\hdots, N_t^n\}$, $\I^{n,*}:=\I^n\setminus \{N_{t}^{n}\}$,  $t^n_k=k\Delta t_n\;(k\in \I^n)$, and $\G^n=\{i\Delta x_n \, | \, i\in\ZZ^d\}$.  We assume that $N_s^{n}\geq N_t^{n}$. For $k\in \I^{n,*}$ and $x\in \G^n$, we denote by $\SS_{k+1}^{1,n}(x)$, $\SS^{2,n}_{k+1}(x,\mathrm{y}_1)$ ($\mathrm{y}_1\in \SS^{1,n}_{k+1}(x)$),  and $\SS_{k+1}^n(x)$ the sets defined in \eqref{definicion_de_los_S} associated with $\Delta t_n$ and $\Delta x_n$. For $k\in \I^n$, the set $\SS^n_{k}$ is defined as in \eqref{eq:family_s}.  Denote by $\Gamma^n$ the set of continuous functions $\gamma\colon [0,T]\to \cR^d$ such that for each $k \in \I^n$, 
$\gamma(t^n_k)\in \SS_{k}^n$ and, for every $k\in\I^{n,*}$,  the restriction of $\gamma$ to the interval $[t^n_{k}, t^n_{k+1}] $ is affine. Finally, let $M^{n}\in\M$ be a solution to Problem~\ref{MFG-SL_discrete_scheme_n-degenerate} associated with the previous parameters and, recalling \eqref{eq:def-P_k_M}, let us define $\xi^n\in \P(\Gamma)$ as
\be 
\label{eq:def-xi-n}
 \xi^n= \sum_{\gamma \in \Gamma^{n}} M_0^n(\gamma(0))P^n(\gamma)\delta_{\gamma} \in \P(\Gamma), \quad \text{where} \quad P^{n}(\gamma):=  \prod_{k=0}^{N^n_t-1}P^{M^n}_k(\gamma(t^n_k), \gamma(t^n_{k+1})). 
\ee

We extend $M^n$ to the element in $C([0,T];\P_{1}(\RR^{d}))$ defined by
\be
\label{rem:def-Mn-cont}
[0,T] \ni t \mapsto M^n(t):=e_t\sharp \xi^n \in  \P_1(\RR^d). 
\ee
\begin{lemma} 
\label{lem:compacidad}
Assume that~\ref{h:h1},~\ref{h:h2}, and~~\ref{h:h3} are in force. Then the following hold:
\begin{enumerate}[label={\rm(\roman*)}]
\item 
\label{lem:compacidad_i}
The family $\xi^{n}$ has at least one accumulation point in $\P(\Gamma)$.
\item 
\label{lem:compacidad_ii}
The family $M^{n}$ has at least one accumulation point in $C([0,T];\P_{1}(\RR^{d}))$.
\end{enumerate}
\end{lemma}
\begin{proof}
\ref{lem:compacidad_i}: Since $ \supp(\xi^n)\subset \Gamma^n$, it follows from Lemma~\ref{the_big_compact} that there exists $C_{\infty}>0$ such that
\be
\label{eq:cota_C_infty_trayectoria}
\|\gamma \|_{\infty} \leq C_{\infty}\quad\text{for all }\gamma\in \supp(\xi^n).
\ee
Moreover, if $\gamma \in \supp(\xi^n)$, then $\gamma$ is absolutely continuous with 
\be
\label{eq:expresion_Derivatives}
\dot{\gamma}(t)=  \frac{\gamma(t_{k+1}^n)-\gamma(t_{k}^n) }{\Delta t_n} \quad \text{for all }k\in \I^{n,*},\,t\in]t_{k}^n, t_{k+1}^n[.
\ee
Writing $\gamma(t_k^n) =\left(\gamma_1(t_k^n) ,\gamma_2(t_k^n) \right)\in \cR^r\times \cR^{d-r} $, the definition of $\SS_{k+1}^{1,n}(\gamma(t_{k}^{n}))$ and~\eqref{eq:cota_C_infty_trayectoria} yield
$$
\gamma_1(t_{k+1}^n) =\gamma_1(t_{k}^n)+ \Delta t_n \left[ A_1(t_k^n, \gamma(t_{k}^n)) + B_1(t_k^n)\alpha(k,\gamma(t_{k}^n),\gamma_1(t_{k+1}^n))  \right],
$$
with $\left|\alpha(k,\gamma(t_{k}^n),\gamma_1(t_{k+1}^n))\right| \leq C_{{\rm b}}(1 + C_\infty)$. Thus, using \eqref{eq:cota_A} we deduce that 
$$
\left| \frac{\gamma_1(t^n_{k+1})-\gamma_1(t^n_k)}{\Delta t_n} \right|\leq \left(C_A+ C_B C_{{\rm b}}\right)\left(1+C_\infty\right),
$$
and, by \eqref{definition_alpha_n} and \eqref{definicion_de_los_S}, we obtain
\be 
\label{eq:cota_infinito_derivada}
 \left| \frac{\gamma_2(t^n_{k+1})-\gamma_2(t^n_k)}{\Delta t_n} \right|\leq  C_I\frac{\Delta x_n}{\Delta t_n}+\left(C_A+ C_B C_{{\rm b}}\right)\left(1+C_\infty\right).
\ee
Since $\Delta x_{n}\leq \Delta t_{n}$, we deduce from~\eqref{eq:expresion_Derivatives} that there exists $D_\infty>0$ such that 
$$
\|\dot{\gamma} \|_{\infty}\leq D_{\infty}\quad\text{for all }\gamma\in \supp(\xi^n) 
$$
and hence $\supp(\xi^n)\subset\{ \gamma \in W^{1,\infty}([0,T]; \RR^d) \, | \, \| \gamma\|_{\infty} \leq C_{\infty}, \; \|\dot{\gamma}\|_{\infty} \leq D_{\infty} \}$, which is a compact subset of $(\Gamma,\|\cdot\|_{\infty})$. Thus, the result follows from Prokhorov's theorem (see e.g.~\cite[Theorem~5.1.3]{ambrosio2008gradient}).

\ref{lem:compacidad_ii}: By~\eqref{eq:cota_C_infty_trayectoria}, for every $t\in [0,T]$ and $n\in\NN$, we have
$$
\supp\left(M^n(t)\right)\subset \ov{\mathrm{B}}_\infty(0,C_{\infty}),
$$
and, by~\eqref{def:wasserstein_uno} and~\eqref{eq:cota_infinito_derivada}, 
$$
d_1\left(M^n(s), M^n(t)\right)\leq D_{\infty} |s-t|\quad\text{for all }s,\,t\in [0,T],\,n\in\NN.
$$
Since $\{\mu \in \P_1(\RR^d) \; | \; \supp(\mu) \subset  \ov{\mathrm{B}}_\infty(0,C_{\infty}) \}$ is compact in $\P_1(\RR^d)$ (see e.g. \cite[Proposition 7.1.5]{ambrosio2008gradient}), the result follows from the Arzel\`a-Ascoli theorem.
\end{proof}

Using the previous compactness result and arguing as in the proof of \cite[Theorem~5.1]{Gianatti-Silva-arxiv}, one obtains the following convergence result.
\begin{theorem}   
\label{main_result}
Assume that \ref{h:h1},~\ref{h:h2}, and ~\ref{h:h3}  hold and that, as $n\to \infty$, $N_t^{n} \to \infty$, $N^n_s \to \infty$, $   N^n_t/N^n_s \to 0$, and $\eps_n=o\left( 1/(N^n_t \log(N^n_s)) \right)$. Then there exists a solution $\xi^{*}$ to Problem~\ref{mfg_problem} such that, up to some subsequence, $\xi^{n}\to\xi^{*}$ narrowly in $\P(\Gamma)$ and $M^{n}\to m^*:= e_{(\cdot)}\sharp\xi^{*}$ in $C([0,T];\P_1(\RR^{d}))$.

In addition, if~\ref{h:h4} holds and for every $m\in C([0,T];\P_{1}(\RR^d))$ and $m_{0}^{*}$-a.e. $x\in\RR^d$ problem~\eqref{oc_problem} admits a unique solution, then the whole sequence $(\xi^{n})_{n\in\NN}$ converges narrowly towards the unique solution to Problem~\ref{mfg_problem}.
\end{theorem}
%----------------------------------------------------------------------------------------------------
\section{Numerical results}
\label{sec:numerical_result}

In this section we implement our numerical method in two examples. For computational simplicity, we consider here one-dimensional problems, i.e. $d=1$, and dynamics~\eqref{eq:state_equation} having the form $\dot{\gamma}=\alpha$. We refer the reader to \cite[Example 2]{Gianatti-Silva-arxiv} for the implementation of the scheme in a two-dimensional example, where a typical agent controls its acceleration and the cost functional satisfies~\ref{h:h1}.  We focus our attention on cost functionals satisfying~\ref{h:h1} and~\ref{h:h4}, with $\ell_0(t,a,x)$ having
polynomial growth on $(a,x)$, and $f$ and $g$ begin given by 
\be 
\ba{rcl}
\ds f(t,x,\mu) &=& \ds \theta_1 (\rho_\sigma\star \mu)(x) \quad \text{for all } (t,x,\mu)\in [0,T]\times\RR\times \P_1(\RR), \\[6pt] 
\ds g(x,\mu)&=&g_0(x) + \theta_2 (\rho_\sigma\star \mu)(x) \quad \text{for all } (x,\mu)\in\RR\times \P_1(\RR), 
\ea
\ee
where $\theta_{1}$, $\theta_{2}\in [0,\infty[$,  $\sigma\in ]0,\infty[$, $g_{0}\colon\RR\to\RR$ satisfies~\ref{h:h1}, and
\be 
\label{rho_sigma}
\rho_\sigma(x):= \frac{1}{\sqrt{2\pi}\sigma}e^{-x^2/2\sigma^2} \quad \text{for all }x\in\RR.
\ee
Notice that the convolution terms in $f$ and $g$, which model the aversion of a typical player to crowded areas, satisfy the monotonicity condition in~\ref{h:h4}\ref{h:h4_ii}.

Let $(\Delta t,\Delta x)\in ]0,\infty[^2$ and $\eps>0$. Under the assumptions above, the finite MFG Problem~\ref{MFG-SL_discrete_scheme_n-degenerate} associated with these parameters admits a unique solution $M^{*}\in\M$. In order to approximate $M^{*}$, we consider the {\it fictitious play} sequence 
$$
\ov{\mathsf{M}}^0\in \M \; \; \text{arbitrary}, \;  \quad (\forall\,n\geq 1) \quad \mathsf{M}^{n+1}= {\bf br}(\ov{\mathsf{M}}^{n}), \quad \ov{\mathsf{M}}^{n+1}= \frac{n}{n+1}\ov{\mathsf{M}}^{n} + \frac{1}{n+1}\mathsf{M}^{n+1},
$$
which, by \cite[Theorem~3.2]{MR4030259}, satisfies 
$(\mathsf{M}^{n}, \ov{\mathsf{M}}^{n}) \underset{n\to\infty}{\longrightarrow} (M^*, M^*)$. In the tests below, setting
$$
|{\bf br}(\ov{\mathsf{M}}^n)-\ov{\mathsf{M}}^n|_{L^1}:= \frac{1}{N_t+1} \sum_{k=0}^{N_t}\sum_{x\in \mathcal{S}_{k}} |{\bf br}(\ov{\mathsf{M}}^n)_{k}(x)-\ov{\mathsf{M}}_k^n(x)|
$$
and given a tolerance parameter $\delta>0$, we implement the following  fictitious play algorithm:
\vspace{0.5cm}

\begin{algorithm}[H]
\label{alg:1}
\SetAlgoLined
\KwData{ $\mathsf{M}^0\in\M$, $\delta>0$}
$e\leftarrow \delta +1$\\[3pt]
$n\leftarrow 1$\\[3pt]
$\ov{\mathsf{M}}^1\leftarrow \mathsf{M}^0$

\While{$e>\delta$}{\vspace{0.2cm}
$\mathsf{M}^{n+1}= {\bf br}(\ov{\mathsf{M}}^n)$\\[3pt]
$e\leftarrow  |\mathsf{M}^{n+1}-\ov{\mathsf{M}}^n|_{L^1}$\\[3pt]
$\ov{\mathsf{M}}^{n+1}= \frac{n}{n+1}\ov{\mathsf{M}}^n+\frac{1}{n+1}\mathsf{M}^{n+1}$\\[3pt]
$n\leftarrow n+1$ 
}
\Return{$\ov{\mathsf{M}}^{n-1}$}
 \label{alg:FP}
\end{algorithm}
\vspace{0.5cm}

In both examples below, we consider a time horizon $T=1$, $\sigma = 0.07$, and 
$$
\Delta t = 1/30, \quad \Delta x = 1/150,\quad \text{and}\quad \eps = 0.002.
$$

Given an initial distribution $m_0^{*}\in\P_1(\RR)$, we initialize the fictitious play algorithm by defining $\mathsf{M}^0\in\M$ with constant time marginals given by $\mathsf{M}^0_k=M_0$, for $k=1, \dots, 30$, where $M_0$ is obtained by discretizing the initial distribution $m_0^{*}$ according to \eqref{def-br}. As it was mentioned above, the algorithm converges for an arbitrary initial condition $\mathsf{M}^0\in\M$. However, since the term $\mathsf{M}^0/n$ is involved in the computation of $(\ov{\mathsf{M}}^{n}, \mathsf{M}^{n+1})$ the convergence of the algorithm could be slow. To accelerate the method we can update the initial condition when some tolerance is achieved, i.e. we use as the initial distribution the approximated equilibrium $\ov{\mathsf{M}}^{n-1}$ obtained for a given tolerance parameter $\delta>0$, and then we run the algorithm for a smaller tolerance parameter. In our tests, we update the initial condition twice, taking the tolerance parameters $\delta_1=0.1$ and $\delta_2=0.01$. We stop the algorithm when the tolerance $\delta_3=0.001$ is reached.  

\subsection{Example 1} 
We consider an absolutely continuous initial distribution $m_0^{*}\in\P_1(\RR)$ given by
$$
\dd m_0^{*}(x)= \mathbb{I}_{[-1,1]}(x)\frac{e^{-x^2/0.04}}{\int_{-1}^1 e^{-y^2/0.04}\dd y}\dd x\quad\text{for all }x\in\RR,
$$
where $\mathbb{I}_{[-1,1]}(x)=1$ if $x\in[-1,1]$ and $\mathbb{I}_{[-1,1]}(x)=0$, otherwise.  Given $\zeta_1$, $\zeta_2\in\RR$, we define
$$
\ell(t,a,x,\mu)=\frac{|a|^4}{4}+\zeta_1 |x-0.4|^2 |x+0.7|^2+\theta_1 f(x,\mu) \quad \text{and}\quad g(x,\mu)=\zeta_2 |x-0.4|^2 |x+0.7|^2+\theta_2 f(x,\mu).
$$
Notice that the functions $\ell$ and $g$ satisfy \ref{h:h1} for $p=4$. We run our algorithm for different values of $(\zeta_1,\zeta_2,\theta_1,\theta_2)$. In Figure~\ref{fig:example1} we show the returned distributions for the smallest tolerance parameter $\delta_3$. 
In Table~\ref{tabla1}, we provide the number of iterations needed for attaining the tolerances $\delta_1$, $\delta_2$, and $\delta_3$.

\vspace{-0.1cm}

\begin{table}[h]
\begin{center}
\vspace{0.1cm}
\begin{tabular}{|c|c|c|c|}
	\hline
\;\;$(\zeta_1,\zeta_2,\theta_1,\theta_2)$\;\; &\;\; $\delta_1=0.1$\;\; & \;\;$\delta_2=0.01$ \;\;&\;\; $\delta_3=0.001$\;\;\\[2pt]
	\hline

$(1,1,1,1)$ &$14$ & $10$ & $9$\\
$(1,1,1,5)$ & $10$& $12$ & $9$\\
$(5,1,1,1)$ & $18$ & $11$ & $9$\\
$(1,0,1,0)$ & $8$ & $7$ & $6$\\
\hline
\end{tabular}
\end{center}
\vspace{-0.2cm}
\caption{Number of iterations to obtain the desired accuracies.}
\label{tabla1}
\end{table}

\vspace{-0.2cm}

\begin{figure}[h]
\centering
\begin{tabular}{cccc}
\subfloat[ {\tiny $(\zeta_1,\zeta_2,\theta_1,\theta_2)=(1,1,1,1)$.}]{\includegraphics[width=.42\textwidth]{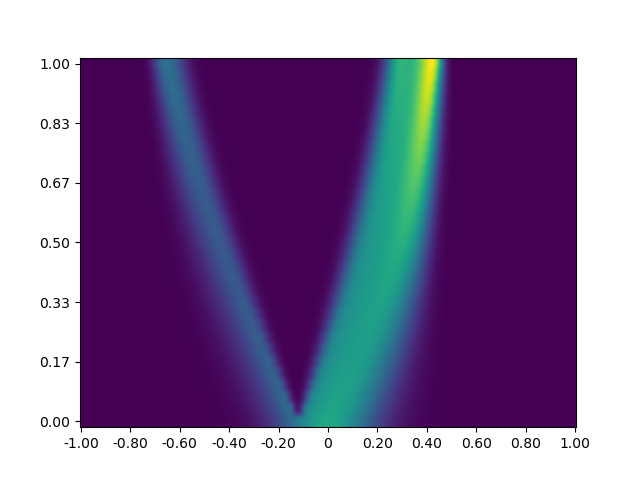}}& 
\subfloat[\tiny$(\zeta_1,\zeta_2,\theta_1,\theta_2)=(1,1,1,5)$.]{\includegraphics[width=.42\textwidth]{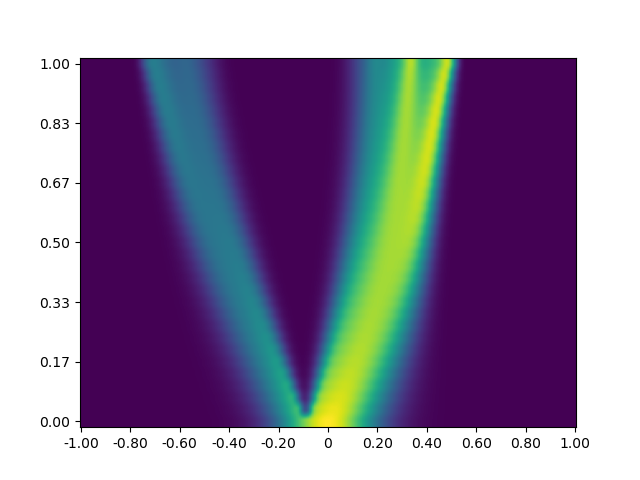}}\\ 
\subfloat[ \tiny$(\zeta_1,\zeta_2,\theta_1,\theta_2)=(5,1,1,1)$.]{\includegraphics[width=.42\textwidth]{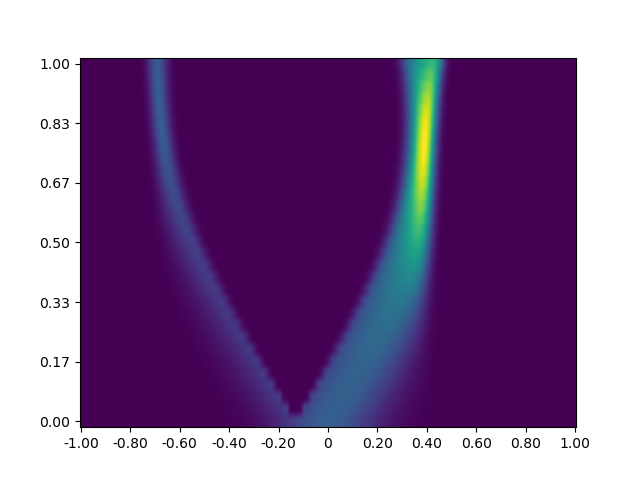}}& 
\subfloat[ \tiny$(\zeta_1,\zeta_2,\theta_1,\theta_2)=(1,0,1,0)$.]{\includegraphics[width=.42\textwidth]{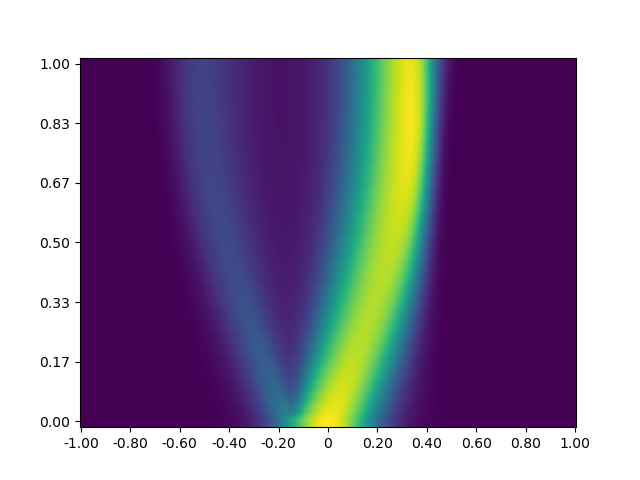}}
\end{tabular}
\vspace{-0.2cm}
\caption{Approximate equilibria for the tests  in Example 1}
\label{fig:example1}
\end{figure}

In this example, the initial distribution is concentrated around $x=0$. The cost functional penalizes the distance to the points $-0.7$ and $0.4$, large values of the speed, and incites a typical player to avoid crowded regions. Since $0.4$ is closer to zero, we see that most of the agents tend to concentrate around that point and, as the congestion term becomes more important, we see how the agents tend to separate from each other.

\subsection{Example 2}

The initial distribution $m_0^{*}\in\P_1(\RR)$ is given by
$$
\dd m_0^{*}(x)= \mathbb{I}_{[-1,1]}(x)\frac{e^{-(x-0.2)^2/0.01}+e^{-(x+0.2)^2/0.01}}{\int_{-1}^1 (e^{-(y-0.2)^2/0.01} +  e^{-(y+0.2)^2/0.01})\dd y}\dd x\quad\text{for all }x\in\RR.
$$
Given $\zeta_1$, $\zeta_2\in\RR$, we define
$$
\ell(t,a,x,\mu)=\frac{|a|^4}{4}+\zeta_1 |x-0.6|^2 |x+0.2|^2+\theta_1 f(x,\mu) \quad \text{and}\quad g(x,\mu)=\zeta_2 |x-0.6|^2 |x+0.2|^2+\theta_2 f(x,\mu).
$$
We consider the same parameters as those in Example 1 and we display in Figure~\ref{fig:example2} the distributions obtained for the smallest tolerance $\delta_3=0.001$. In Table~\ref{tabla2}, we show the number of iterations required to reach the tolerances $\delta_1$, $\delta_2$, and $\delta_3$.

\vspace{-0.2cm}

\begin{table}[h]
\begin{center}
\vspace{0.1cm}
\begin{tabular}{|c|c|c|c|}
	\hline
\;\;$(\zeta_1,\zeta_2,\theta_1,\theta_2)$\;\; &\;\; $\delta_1=0.1$\;\; & \;\;$\delta_2=0.01$ \;\;&\;\; $\delta_3=0.001$\;\;\\[2pt]
	\hline

$(1,1,1,1)$ &$7$ & $6$ & $6$\\
$(1,1,1,5)$ & $10$& $12$ & $9$\\
$(5,1,1,1)$ & $12$ & $10$ & $9$\\
$(1,0,1,0)$ & $4$ & $6$ & $6$\\
\hline
\end{tabular}
\end{center}
\vspace{-0.2cm} \caption{Number of iterations to obtain the desired accuracies.}
\label{tabla2}
\end{table}

\vspace{-0.2cm}

\begin{figure}[h]
\centering
\begin{tabular}{cccc}
\subfloat[\tiny $(\zeta_1,\zeta_2,\theta_1,\theta_2)=(1,1,1,1)$.]{\includegraphics[width=.42\textwidth]{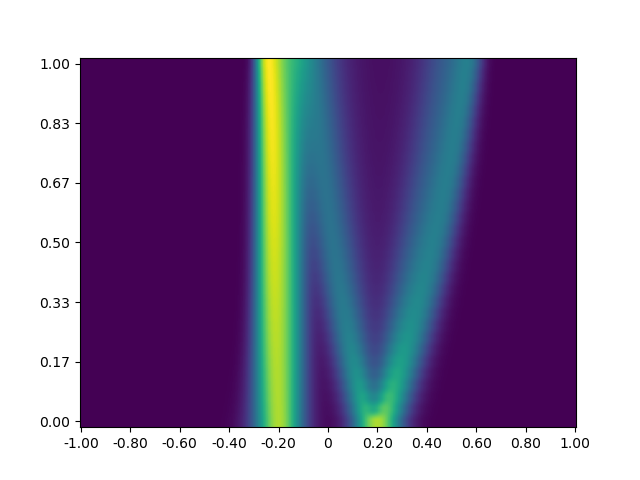}}& 
\subfloat[ \tiny $(\zeta_1,\zeta_2,\theta_1,\theta_2)=(1,1,1,5)$.]{\includegraphics[width=.42\textwidth]{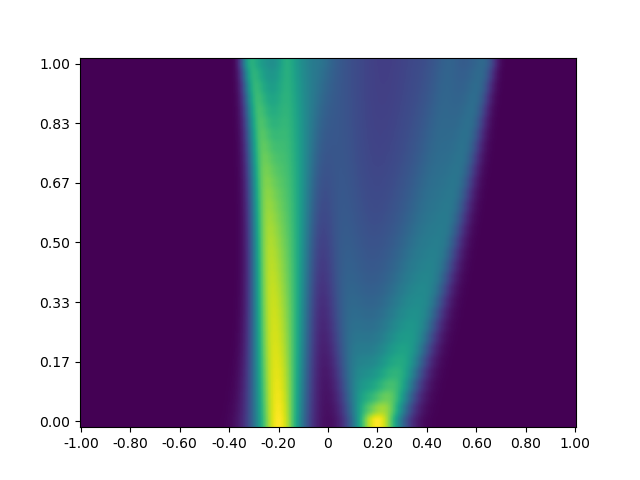}} \label{caso_B}\\ 
\subfloat[ \tiny $(\zeta_1,\zeta_2,\theta_1,\theta_2)=(5,1,1,1)$.]{\includegraphics[width=.42\textwidth]{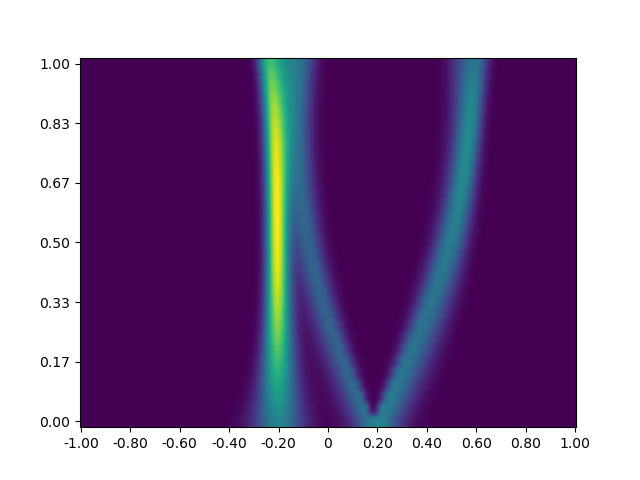}}& 
\subfloat[ \tiny $(\zeta_1,\zeta_2,\theta_1,\theta_2)=(1,0,1,0)$.]{\includegraphics[width=.42\textwidth]{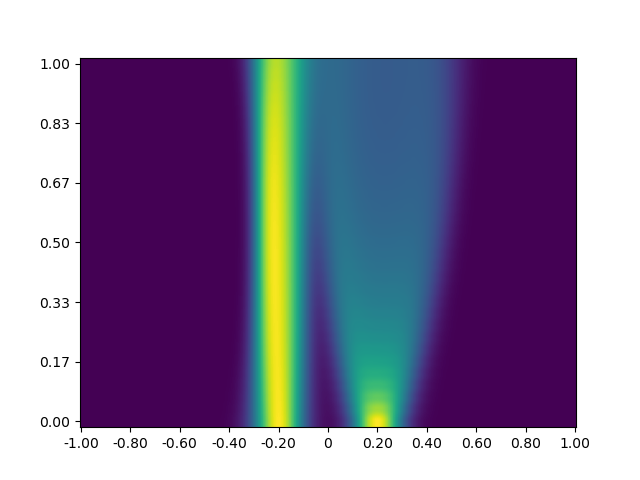}}
\end{tabular}
\vspace{-0.2cm}
\caption{Approximate equilibria for the tests  in Example 2}
\label{fig:example2}
\end{figure}

In this example, we start with a distribution concentrated around the points $-0.2$ and $0.2$. The cost functional penalizes the distance to the points $-0.2$ and $0.6$. Although these two points are symmetric with respect to $0.2$, we see that, in order to avoid crowded regions, more agents concentrated around $0.2$ at $t=0$ tend to go towards $0.6$ instead of $-0.2$. Once again we observe the impact of the congestion terms in the final distributions of the agents. 

For a better understanding of the fictitious play method, we end this section by displaying in Figure~\ref{fig:stepbystep} the first iterations of the algorithm when $(\zeta_1,\zeta_2,\theta_1,\theta_2)=(1,1,1,5)$. The final distribution in this case is shown in the top right corner of Figure~\ref{fig:example2}.

\begin{figure}[h]
\centering
\begin{tabular}{cccc}
\subfloat[ $\ov{\mathsf{M}}^1={\mathsf{M}}^0$.]{\includegraphics[width=.3\textwidth]{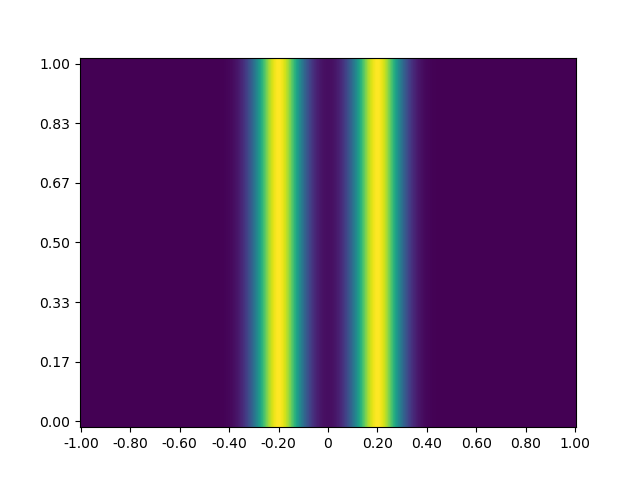}}& 
\subfloat[ ${\bf br}(\ov{\mathsf{M}}^1)$.]{\includegraphics[width=.3\textwidth]{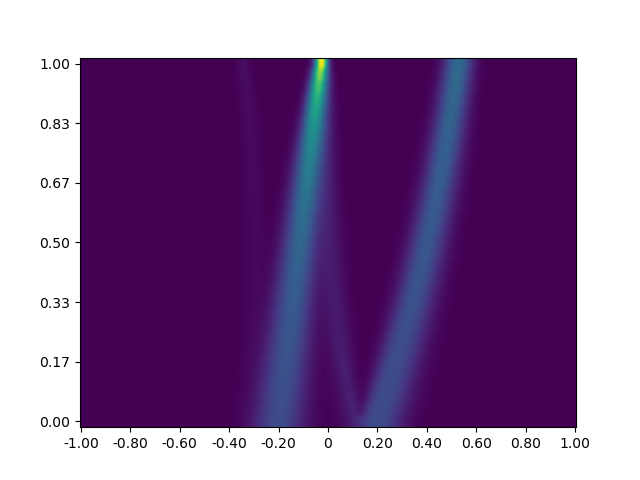}}\\ 
\subfloat[ $\ov{\mathsf{M}}^2$.]{\includegraphics[width=.3\textwidth]{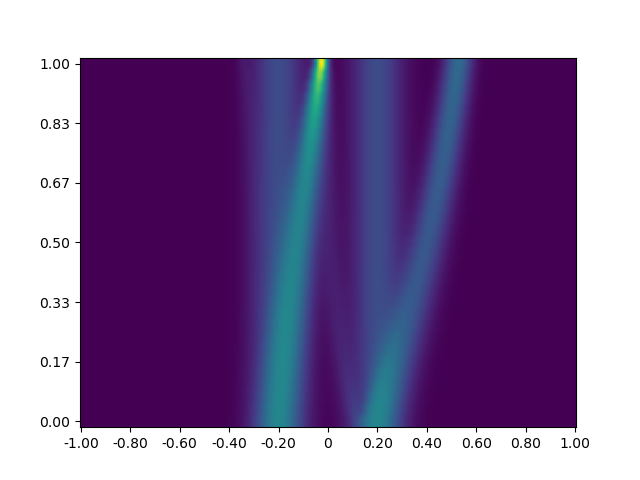}}& 
\subfloat[ ${\bf br}(\ov{\mathsf{M}}^2)$.]{\includegraphics[width=.3\textwidth]{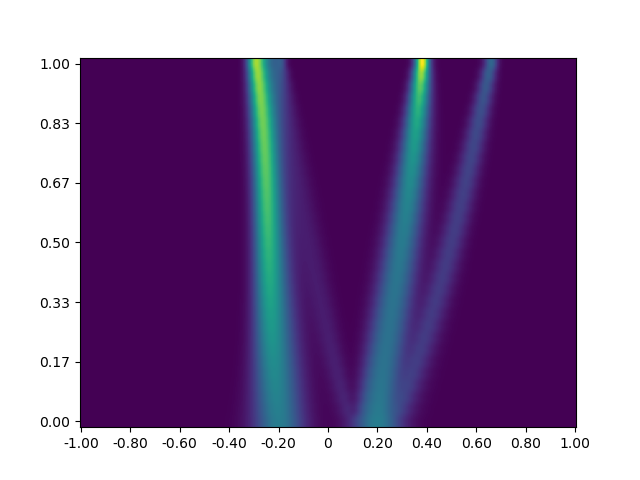}}\\
\subfloat[ $\ov{\mathsf{M}}^3$.]{\includegraphics[width=.3\textwidth]{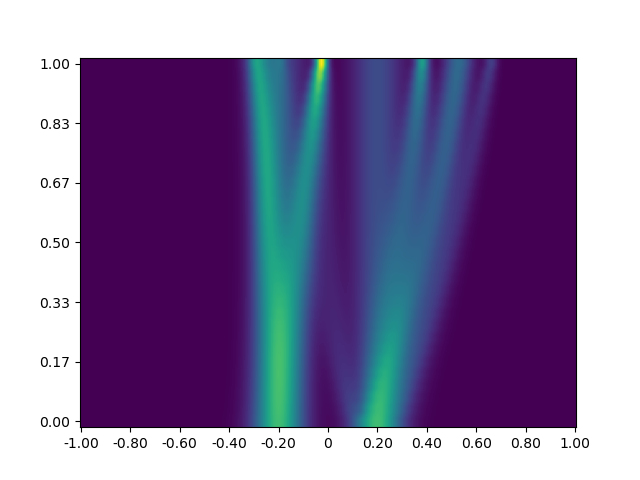}}& 
\subfloat[ ${\bf br}(\ov{\mathsf{M}}^3)$.]{\includegraphics[width=.3\textwidth]{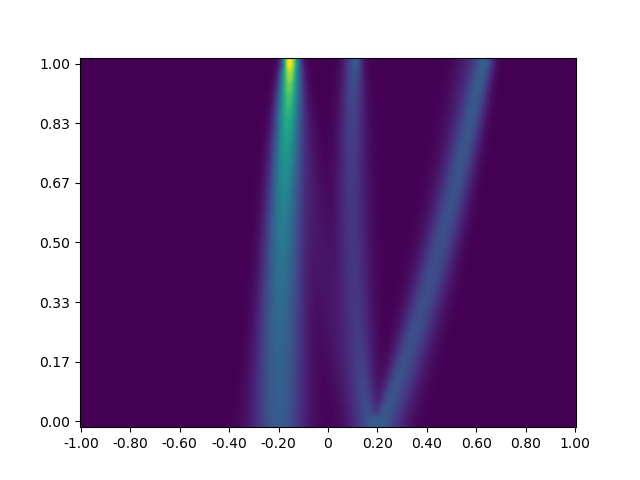}}\\
\subfloat[ $\ov{\mathsf{M}}^4$.]{\includegraphics[width=.3\textwidth]{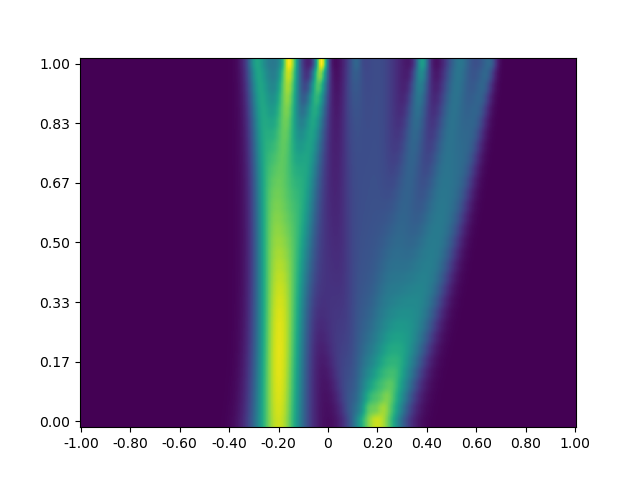}}& 
\subfloat[ ${\bf br}(\ov{\mathsf{M}}^4)$.]{\includegraphics[width=.3\textwidth]{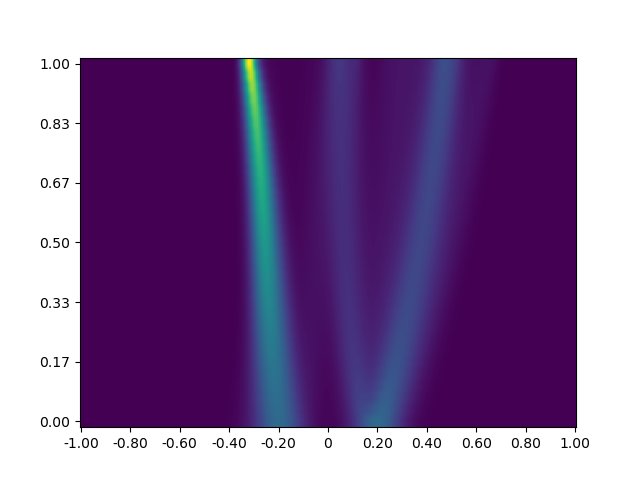}}
\end{tabular}
\caption{First iterations of the algorithm for $(\zeta_1,\zeta_2,\theta_1,\theta_2)=(1,1,1,5)$.}
\label{fig:stepbystep}
\end{figure} 
\bibliographystyle{plain}
\bibliography{hjb,mfg}
\end{document}

%% file: mymacro.tex
\def\dd{{\rm d}}

\newcommand{\ov}[1]{\overline{#1}}

\def\weight(#1,#2){c_{#1,#2}}
 \if {
  
 } \fi
\if {

}{{\caption}}
} \fi

\def\B{\mathcal{B}}

\def\G{\mathcal{G}}

\def\I{\mathcal{I}}

\def\M{\mathcal{M}}

\def\P{\mathcal{P}}

\def\SS{\mathcal{S}}

\def\V{\mathcal{V}}

\def\eps{\varepsilon}

\def\supp{\mathop{\rm supp}}

\def\1B{{\bf  1}}

\newcommand{\NN}{\mathbb{N}}
\newcommand{\ZZ}{\mathbb{Z}}

\newcommand{\QQ}{\mathbb{Q}}

\newcommand{\RR}{\mathbb{R}}

\def\cR{\mathbb{R}}

\newcommand\be{\begin{equation}}
\newcommand\ee{\end{equation}}
\newcommand\ba{\begin{array}}
\newcommand\ea{\end{array}}
\newcommand{\bean}{\begin{eqnarray*}}
\newcommand{\eean}{\end{eqnarray*}}

\def\ds{\displaystyle}